\documentclass[12pt]{article}

\usepackage[utf8]{inputenc}

\usepackage[a4paper,nomarginpar]{geometry}
\usepackage[english]{babel}

\usepackage{amsmath,amsfonts,amssymb,amsthm}
\usepackage{amscd}
\usepackage[pdftex,usenames,dvipsnames]{color}
\usepackage{xcolor}

\newtheorem{theorem}{Theorem}

\newtheorem{proposition}[theorem]{Proposition}
\newtheorem{corollary}[theorem]{Corollary}

\theoremstyle{definition}
\newtheorem{definition}[theorem]{Definition}
\newtheorem{example}[theorem]{Example}

\theoremstyle{remark}
\newtheorem{remark}[theorem]{Remark}

\newcommand{\N}{\mathbb{N}} 
\newcommand{\Z}{\mathbb{Z}} 
\newcommand{\R}{\mathbb{R}} 
\newcommand{\K}{\mathbb{K}} 

\newcommand{\A}{\mathcal{A}}

\newcommand{\abs}[1]{\left\lvert#1\right\rvert}

\newcommand{\orb}{\operatorname{Orb}}

\begin{document}

\title{Recurrence properties of hypercyclic operators
 \thanks{This work is supported in part by MICINN and FEDER, Project MTM2013-47093-P, and by GVA, Project  PROMETEOII/2013/013. The second author is supported by a grant of FRIA.}}


\author{J. B\`{e}s \footnote{Department of Mathematics and Statistics, Bowling Green State University, Bowling Green, OH 43403, USA. e-mail:jbes@math.bgsu.edu },  Q. Menet\footnote{D\'{e}partement de Math\'{e}matique, Universit\'{e} de Mons, 20 Place du Parc, 7000 Mons,
Belgique. e-mail: Quentin.Menet@umons.ac.be}, A. Peris\footnote{IUMPA, Universitat Polit\`{e}cnica de Val\`{e}ncia, Departament de Matem\`{a}tica Aplicada, Edifici 7A, 46022 Val\`{e}ncia, Spain. e-mail: aperis@mat.upv.es} and Y. Puig\footnote{Universit\`{a} degli studi di Milano, Dipartimento di Matematica "Federigo Enriques", Via Saldini 50 - 20133 Milano, Italy. e-mail: yunied.puig@unimi.it}}

\date{ }

\maketitle

\begin{abstract}
We generalize the notions of hypercyclic operators, $\mathfrak{U}$-frequently hypercyclic operators and frequently hypercyclic operators by introducing a new notion of hypercyclicity, called $\A$-frequent hypercyclicity. We then state an $\A$-Frequent Hypercyclicity Criterion, inspired from the Hypercyclicity Criterion and the Frequent Hypercyclicity Criterion, and we show that this criterion characterizes the $\A$-frequent hypercyclicity for weighted shifts. We finish by investigating which kind of properties of density can have the sets ${N(x, U)=\{n\in \N:T^nx\in U\}}$ for a given hypercyclic operator and study the new notion of reiteratively hypercyclic operators.
\end{abstract}


\section{Introduction}

Our purpose is to study how `often' the orbit of a vector under a linear operator can meet arbitrary non-empty open sets. The (chaotic) dynamics of linear operators is a very active topic of research, and the books \cite{BM09} and \cite{GEP11} contain many of the recent advances.

Let $X$ be a separable $F$-space (i.e., a metrizable and complete topological vector space), $L(X)$ the space of continuous linear operators on $X$, and $T\in L(X)$. Given $x\in X$, its \emph{orbit} under $T$ is $\orb (x,T)=\{ x,Tx,T^2x,\dots \}$. We denote by $\N$ the set of positive integers, by $\Z_+$ the set of non-negative integers and for any $x\in X$, any subset $V\subset Y$, we denote the times set of $\orb (x,T)$ hitting $V$ by
$$
N(x,V)=\{n\ge 0:T^nx\in V\}.
$$
Given $A\subseteq \Z_+$, its \emph{upper and lower densities} are defined, respectively, by
    \[
    \overline{d}(A)=\limsup_{k\to\infty}\frac{|A\cap [0,k]|}{k}  \ \ \mbox{ and } \ \ 
     \underline{d}(A)=\liminf_{k\to\infty}\frac{|A\cap [0,k]|}{k}.
    \]

The operator $T$ is said to be \emph{hypercyclic} if there exists $x\in X$ such that for any non-empty open set $V\subset X$, $N(x,V)$ is non-empty (or, equivalently since $X$ has no isolated points, infinite). It is $\mathfrak{U}$-\emph{frequently hypercyclic} if there exists $x\in X$ such that for any non-empty open set $V$, $N(x,V)$ is a set of positive upper density, and
it is \emph{frequently hypercyclic} if there exists $x\in X$ such that for any non-empty open set $V\subset Y$, $N(x,V)$ is a set of positive lower density. Frequently hypercyclic operators were introduced by Bayart and Grivaux \cite{Bayart1} while $\mathfrak{U}$-frequently hypercyclic operators were introduced by Shkarin \cite{Shkarin}.

We generalize these notions of hypercyclicity as follows:

\begin{definition}\label{univset}
Let $\A\subset \mathcal{P}(\Z_+)$ be a non-trivial hereditarily upward family of subsets of $\Z_+$ (i.e., $\emptyset\notin \A$ and 
 for any $A\in \A$, if $A\subset B$, then $B\in \A$). If, moreover,   
\begin{enumerate}
\item[(*)] $\A$ contains a sequence $(A_k)$ of disjoint sets such that for any $j\in A_k$, any $j'\in A_{k'}$, $j\ne j'$, we have $|j'-j|\ge \max\{k,k'\}$, 
\end{enumerate}
then we say that $\A$ is a \emph{hypercyclicity set}. Given a non-trivial hereditarily upward family $\A\subset \mathcal{P}(\Z_+)$ and $T\in L(X)$, the operator  $T$ is called 
\emph{$\A$-frequently hypercyclic} if there exists $x\in X$ such that for any non-empty open set $V\subset X$, $N(x,V)\in \A$. Such a vector $x$ is called an $\A$-frequently hypercyclic 
vector for $T$.
\end{definition}

An operator $T$ is thus hypercyclic, $\mathfrak{U}$-frequently hypercyclic or frequently hypercyclic if it is $\A$-frequently hypercyclic, respectively, for
$\A_\infty$ (the family of infinite subsets of $\Z_+$), for the set of positive upper density sets, or for the set of positive lower density sets. This concept also generalizes the notions of $(m_k)$-hypercyclic operators introduced by Bayart and Matheron~\cite{Bayart3}.

For  hypercyclic operators and frequently hypercyclic operators, we have at our disposal the well-known Hypercyclicity Criterion~\cite{Bes} and the Frequent Hypercyclicity Criterion~\cite{4Bonilla}. We generalize these criteria to the framework of $\A$-frequently hypercyclic operators (Section~\ref{Sec: Crit}). In particular, we obtain a $\mathfrak{U}$-Frequent Hypercyclicity Criterion. The $\A$-Frequent Hypercyclicity Criterion also improves the Frequent Hypercyclicity Criterion. Indeed, we know that the Frequent Hypercyclicity Criterion does not characterize frequently hypercyclic operators because Bayart and Grivaux~\cite{2Bayart1} have exhibited a frequently hypercyclic weighted shift on $c_0$ that is neither chaotic nor mixing, whereas if $T$ satisfies the Frequent Hypercyclicity Criterion, then $T$ is mixing and chaotic. However, we succeed to show that, in the case of weighted shifts on $\ell^p$ or $c_0$, the $\A$-Frequent Hypercyclicity Criterion characterizes the $\A$-frequent hypercyclicity (Section~\ref{SEC Shift}) and thus, in particular, an operator $T$ can satisfy the $\A$-Frequent Hypercyclicity Criterion even if it is neither chaotic nor mixing. This characterization of $\A$-frequently hypercyclic weighted shifts complements the characterization of hypercyclic weighted shifts~\cite{2Grosse3,2Salas} and the recent characterization of frequently hypercyclic weighted shifts on $\ell^p$ and $c_0$~\cite{Bayart2}.

In Section~\ref{dens} we will concentrate in $\A$-frequently hypercyclic operators for families $\A$ with positive Banach densities.

   \begin{definition}
Let $A\subseteq \Z_+$, $\alpha^s:=\limsup_{k\to \infty} |A\cap [k+1, k+s]|$ and $\alpha_s:=\liminf_{k\to \infty} |A\cap [k+1, k+s]|$.
    The \emph{upper and lower Banach densities} are defined respectively by
    \[
    \overline{Bd}(A)=\lim_{s\to\infty}\frac{\alpha^s}{s} \ \ \mbox{ and } \ \ 
     \underline{Bd}(A)=\lim_{s\to\infty}\frac{\alpha_s}{s}.
    \]
       \label{definition.density}
   \end{definition}
The above densities are well-defined for any set $A\subseteq \Z_+$. A proof of this fact can be found in \cite{Salat} and we have
    \begin{equation}
    \label{eq.density}
    \underline{Bd}(A)\leq \underline{d}(A)\leq \overline{d}(A) \leq \overline{Bd}(A). 
    \end{equation}
For any choice of real numbers $0\leq r_1\leq r_2 \leq r_3 \leq r_4 \leq 1$, one can even find sets $A$ such that
 \[
 \underline{Bd}(A)=r_1, \underline{d}(A)=r_2, \overline{d}(A)=r_3, \overline{Bd}(A)=r_4.
 \]
  This result has been recently announced, and will appear in a joint paper by G. Grekos, R. Jin and L. Mi$\check{s}$\'{i}k.

Suppose we want to classify and study all the recurrence properties appearing when $N(x, U)$ has a lower or upper (Banach) density  greater than 0 or equal to 1. We show that only the following properties make sense.
 \begin{itemize}
 \item $\underline{d} \big(N(x, U)\big)>0$, i.e.  {$T$ is frequently hypercyclic}
  \item $\overline{d} \big(N(x, U)\big)>0$, i.e.  {$T$ is $\mathfrak{U}$-frequently hypercyclic}
   \item $\overline{Bd} \big(N(x, U)\big)>0$, i.e. \emph{$T$ is reiteratively hypercyclic}.
 \end{itemize}

We then focus on the new notion of reiteratively hypercyclic operators. Obviously, frequently hypercyclic operators are $\mathfrak{U}$-frequently hypercyclic and these in turn are reiteratively hypercyclic. In a recent article, Bayart and Ruzsa have shown the existence of a $\mathfrak{U}$-frequently hypercyclic weighted backward shift on $c_0$ which is not frequently hypercyclic, see \cite[Theorem~5]{Bayart2}.
We complement this result by showing that reiterative hypercyclicity is not equivalent to $\mathfrak{U}$-frequent hypercyclicity. More precisely, thanks to the characterization of $\mathcal{A}$-frequently hypercyclic weighted shifts, we show that there exists a reiteratively hypercyclic weighted backward shift on $c_0$ which is not $\mathfrak{U}$-frequently hypercyclic (Theorem~\ref{counterexc0}). On the other hand, Bayart and Ruzsa proved that frequent hypercyclicity and $\mathfrak{U}$-frequent hypercyclicity are equivalent notions for backward shifts on $\ell^p(\Z)$ and $\ell^p(\Z_+)$, see \cite[Theorems~3-4]{Bayart2}. We generalize this equivalence by showing that every reiteratively hypercyclic weighted shift on $\ell^p(\Z)$ or on $\ell^p(\Z_+)$ is frequently hypercyclic (Theorem~\ref{caraclp}). In particular, we deduce that there exists some mixing operator which is not reiteratively hypercyclic. Finally, we observe that any reiteratively hypercyclic operator is topologically ergodic, i.e.
$$
N(U, V):=\{n\ge 0 \ ; \ T^nU\cap V\ne \emptyset\}
$$
is syndetic (has bounded gaps) for any $U, V$ non-empty open sets. Since topologically ergodic operators are weakly mixing \cite{GEP10}, we deduce that any reiteratively hypercyclic operator $T$ is weakly mixing.


\section{$\A$-frequently hypercyclic operators}

The definition of $\A$-frequently hypercyclic operators justifies the property (*) of hypercyclicity sets (Definition~\ref{univset}). Indeed, the fact that $\A$ is non-trivial provides that, for any non-empty open set $V$, $N(x,V)$ is at least non-empty. The hereditarily upward condition implies that, for any non-empty open sets $U\subset V$, if $N(x,U)\in \A$, then $N(x,V)\in \A$. In particular, if $X$ is a topological vector space with a countable open basis $(U_n)_{n\ge 1}$, we deduce that in order to prove that a vector $x$ is $\A$-frequently hypercyclic, it is sufficient to prove that $N(x,U_n)\in\A$ for each $n\ge 1$. Finally,  condition (*) is necessary for the existence of $\A$-frequently hypercyclic operators on Banach spaces.

\begin{proposition}\label{propban}
Let $X$ be a separable Banach space, $X\ne \{0\}$, let $\A \subset \mathcal{P}(\Z_+)$ be a non-trivial hereditarily upward family, and let $T\in \mathcal{L}(X)$.
If $T$ is $\A$-frequently hypercyclic, then $\A$ is a hypercyclicity set.
\end{proposition}
\begin{proof}
Since $X\ne \{0\}$ and $T$ is $\A$-frequently hypercyclic, the norm of $T$ is bigger than~$1$. Let $K=\|T\|>1$ and $x$ an $\A$-frequently hypercyclic vector for $T$. We first show that for any $n\ge 1$, any $C>0$, there exist $y\in X$ and $\varepsilon<1$ such that if $A:=N(x,B(0,C))$ and $B:=N(x,B(y,\varepsilon))$, then
$d(A,B)\ge n$ and for any $j\ne j'\in B$, $|j-j'|\ge n$.

If $j\in A$, then $\|T^jx\|<C$ and we have $\|T^ix\|<CK^{n-1}$ for any $j\le i\le j+n-1$. In particular, if $\|y\|>CK^{n-1}+1$, we deduce that for any $j\le i\le j+n-1$, we have $i\notin B$. Therefore, we consider a vector $y\in \orb (x,T)$ such that $\|y\|>CK^{n-1}+1$ and such that for any $0\le i\le n-1$, $\|T^iy\|>C+K^{n-1}$. Such a vector exists because otherwise the orbit of $x$ would not be dense. We thus have that for any $0\le i\le n-1$, if $j\in B$, then
\[\|T^{j+i}x\|\ge \|T^iy\|-\|T^i(T^jx-y)\|>(C+K^{n-1})-K^{n-1}\|T^jx-y\|\ge C.\]
We deduce that for any $0\le i\le n-1$, if $j\in B$, then $j+i\notin A$ and we thus conclude that $d(A,B)\ge n$.

Now, since $y\in \orb (x,T)$, the vector $y$ is not periodic and there thus exists $0<\eta\le 1$ such that for any $1\le i\le n-1$, $\|T^iy-y\|>\eta$. If we let $\varepsilon:=\frac{\eta}{1+K^{n-1}}$, then we have for any $j\in B$, any $1\le i\le n-1$,
\[\|T^{j+i}x-y\|\ge \|T^iy-y\|- \|T^{i}(T^jx-y)\|> \eta-K^{n-1}\varepsilon\ge \varepsilon.\]
Hence, for any $j,j'\in B$, $j\ne j'$, we have $|j-j'|\ge n$.

We conclude that there exists a sequence $(y_k)\subset X$, an increasing sequence $(C_k)$ and a sequence $(\varepsilon_k)$ such that if $A_k:=N(x,B(0,C_k))$ and $B_k:=N(x,B(y_k,\varepsilon_k))$, then for any $k\ge 1$, $d(A_k,B_k)\ge k$, $B_k\subset A_{k+1}$ and for any $j\ne j'\in B_k$, $|j-j'|\ge k$. Since the sequence $(A_k)$ is increasing, we deduce that the sequence $(B_k)\subset \A$ satisfies the desired property. Indeed, if $j\in B_k$, $j'\in B_{k'}$ and $k> k'$, then we know that $j'\in A_k$ and thus $|j-j'|\ge k$, and if $j,j'\in B_{k}$ with $j\ne j'$, we know that $|j-j'|\ge k$ by definition of $B_k$.
\end{proof}

Nevertheless, if $X$ is not a Banach space, it is possible that there exist a family $\A$ and an operator $T$ on $X$ such that $T$ is $\A$-frequently hypercyclic and $\A$ is not a hypercyclicity set.

\begin{example}
Let $\phi:\N\to \N$ such that for any $i,j\ge 1$, there exists $n\ge 1$ satisfying $(\phi(n+1),\phi(n+2))=(i,j)$.
For any $k\ge 1$, we let $\psi(k)=\sum_{l=1}^{k}\phi(l)$, $A_k=\{\psi(n): \phi(n+1)=k\}$ and $\mathcal{A}=\{B:B\supseteq A_k \text{ for some $k$}\}$. Since $d(A_k,A_{k'})=\min(k,k')$, we deduce that $\A$ does not contain a sequence $(B_k)$ of disjoint sets such that for any $j\in B_k$, any $j'\in B_{k'}$, $j\ne j'$, we have $|j'-j|\ge \max\{k,k'\}$. Nevertheless, if we consider the forward shift $S$ on $\omega$ and a dense sequence $(x_k)_{k\ge 1}\subset \omega$ such that for any
$k\ge 1$, \[k> d(x_k):=\sup\{i\ge 0: x_{k}(i)\ne 0\},\] then the vector $x=\sum_{n=1}^{\infty}S^{\psi(n)}x_{\phi(n+1)}$ is an $\A$-frequently hypercyclic vector for the backward shift on $\omega$.
\end{example}

\subsection{$\A$-Frequent Hypercyclicity Criterion}\label{Sec: Crit}
For the hypercyclic operators and the frequently hypercyclic operators on $F$-spaces, we have at our disposal the well-known Hypercyclicity Criterion~\cite{Bes} and the Frequent Hypercyclicity Criterion~\cite{4Bonilla}. We show in this section how to generalize these criteria to the notion of $\A$-frequent hypercyclicity.

We recall that a $F$-space is a topological vector space whose the topology is induced by a complete translation-invariant metric. In fact, if $X$ is a $F$-space, there exists a complete translation-invariant metric $d$ such that $\|x\|=d(x,0)$ is a $F$-norm.
\begin{definition}
Let $X$ be a vector space. A map $\|\cdot\|$ from $X$ to $\R^+$ is a \emph{F-norm} if for any $x,y\in X$, any $\lambda\in \K$,
\begin{enumerate}
\item $\|x+y\|\le \|x\|+\|y\|$;
\item $\|\lambda x\|\le \|x\|$ if $|\lambda|<1$;
\item $\lim_{\lambda\rightarrow 0} \|\lambda x\|=0$;
\item $\|x\|=0$ implies $x=0$.
\end{enumerate}
\end{definition}

An interesting property about $F$-norms following from $1.$ and $2.$ is that for any $x\in X$, any $\lambda\in \K$, we have
\begin{equation}
\|\lambda x\|\le (|\lambda|+1)\|x\|.
\label{7Fnorm}
\end{equation}

The Hypercyclicity Criterion and the Frequent Hypercyclicity Criterion can be stated as follows.

\begin{theorem}[Hypercyclicity Criterion\ \cite{Bes}]
Let $X$ be a separable $F$-space and $T\in \mathcal{L}(X)$.
If there are dense subsets $X_0\subset X$, $Y_0\subset X$, an increasing sequence $(n_k)_{k\ge 1}$ of positive integers and maps $S_{n_k}:Y_0\to X$, such that for any $x\in X_0$, $y\in Y_0$,
\begin{enumerate}
 \item $T^{n_k}x\rightarrow 0$,
 \item $S_{n_k}y\rightarrow 0$,
 \item $T^{n_k}S_{n_k}y\rightarrow y$,
\end{enumerate}
then $T$ is weakly mixing and thus hypercyclic.
\end{theorem}

\begin{theorem}[Frequent Hypercyclicity Criterion\ \cite{4Bonilla}]
Let $X$ be a separable $F$-space and $T\in \mathcal{L}(X)$. If there are a dense subset $Y_0\subset Y$ and $S_n:Y_0\rightarrow X$, $n\ge 0$ such that for each $y\in Y_0$,
\begin{enumerate}
\item $\sum_{n=0}^{\infty}S_ny$ converges unconditionally in $X$,
\item $\sum_{n=0}^{k}T^{k}S_{k-n}y$ converges unconditionally in $Y$, uniformly in $k\ge 0$,
\item $\sum_{n=0}^{\infty}T^{k}S_{k+n}y$ converges unconditionally in $Y$, uniformly in $k\ge 0$,
\item $T^nS_ny\rightarrow y$,
\end{enumerate}
then $T$ is frequently hypercyclic.
\end{theorem}

\begin{remark}
Let $(x_{n,k})_{n\ge 0, k\in I}\subset X$. A collection of series $\sum_{n=0}^{\infty}x_{n,k}$, $k\in I$, is said to be \emph{unconditionally convergent uniformly in} $k\in I$ if for any $\varepsilon>0$, there exists $N\ge 1$ such that for any $k\in I$, any finite set $F\subset [N,\infty[$,
\[\Big\|\sum_{n\in F}x_{k,n}\Big\|<\varepsilon,\]
where $\|\cdot\|$ is a F-norm inducing the topology of $X$.
\end{remark}

We adapt the Hypercyclicity Criterion and the Frequent Hypercyclicity Criterion to the notion of $\A$-frequent hypercyclicity.

\begin{theorem}[$\A$-Frequent Hypercyclicity Criterion]\label{Ahypc}
Let $X$ be a separable $F$-space,  $T\in \mathcal{L}(X)$ and $\A$ a hypercyclicity set. If there exist a dense subset $Y_0\subset X$, $S_n:Y_0\rightarrow X$, $n\ge 0$, and disjoint sets $A_k\in \A$, $k\ge 1$, such that for each $y\in Y_0$,
\begin{enumerate}
\item there exists $k_0\ge 1$ such that $\sum_{n\in A_k}S_n y$ converges unconditionally in $X$, uniformly in $k\ge k_0$,
\item for any $k_0\ge 1$, any $\varepsilon>0$, there exists $k\ge k_0$ such that
for any finite set $F\subset A_k$, any $n\in \bigcup_{l\ge 1} A_l$, $n\notin F$, we have
\[\Big\|\sum_{i\in F}T^nS_iy\Big\|\le \varepsilon,\]
and such that for any $\delta>0$, there exists $l_0\ge 1$ such that for any finite set $F\subset A_k$, any $n\in \bigcup_{l\ge l_0} A_l$, $n\notin F$, we have
\[\Big\|\sum_{i\in F}T^nS_iy\Big\|\le \delta,\]
\item $\sup_{n\in A_k}\|T^nS_ny-y\|\to 0$ as $k\to\infty$,
\end{enumerate}
then $T$ is $\A$-frequently hypercyclic.
\end{theorem}

\begin{proof}
Let $(y_l)\subset Y_0$ be a dense sequence and $\|\cdot\|$ a $F$-norm inducing the topology of $X$. Without loss of generality, we can suppose that $T$ satisfies the $\A$-Frequent Hypercyclicity Criterion for $Y_0=\{y_l:l\ge 1\}$ and for a sequence $(A_k)\subset \A$ of disjoint sets with $A_k\subset [k,\infty[$. Indeed, it suffices to consider a subsequence of the initial sequence $(A_k)$ such that $2.$ remains satisfied for any vector $y_l$.
We then construct recursively an increasing sequence $(k_l)$ such that for any $l\ge 1$,
\begin{enumerate}
\item[i.] for any finite set $F\subset [k_l,\infty[$, any $j\le l$,
\[\Big\|\sum_{n\in F\cap A_{k_j}}S_ny_j\Big\|\le \frac{1}{l2^l},\]
\item[ii.] for any finite set $F\subset A_{k_l}$, any $n\in \bigcup_{k\ge 1} A_k$,
\[\Big\|\sum_{i\in F\backslash\{n\}}T^nS_iy_l\Big\|\le \frac{1}{2^l},\]
\item[iii.] for any $j<l$, any finite set $F\subset A_{k_j}$, any $n\in A_{k_l}$,
\[\Big\|\sum_{i\in F}T^nS_iy_j\Big\|\le \frac{1}{l2^l},\]
\item[iv.] for any $n\in A_{k_l}$,
\[\Big\|T^nS_ny_l-y_l\Big\|\le \frac{1}{2^l}.\]
\end{enumerate}
Property i. follows from $1.$, property ii. and property iii. from $2.$, and property iv. from $3$. Moreover, since we suppose $A_k\subset [k,\infty[$, property~i. implies that for any finite set $F$, for any $j\ge 1$, we have
\begin{equation}
\Big\|\sum_{n\in F\cap A_{k_j}}S_ny_j\Big\|\le\frac{1}{j2^j}.
\label{eq a}\end{equation}
We write $A:=\bigcup_{l}A_{k_l}$ and for any $n\in A$, we let $z_n=y_l$ if $n\in A_{k_l}$. We then consider the vector $x:=\sum_{n\in A}S_nz_n$. We show that $x$ is well-defined and that $x$ is an $\A$-frequently hypercyclic vector. Let $l\ge 1$ and $F$ be a finite subset of $[k_l,\infty[$.
We deduce from i. and \eqref{eq a} that
\begin{align*}
\Big\|\sum_{n\in F\cap A}S_nz_n\Big\|&\le \sum_{j=1}^{l}\Big\|\sum_{n\in F\cap A_{k_j}}S_ny_j\Big\|+\sum_{j=l+1}^{\infty}\Big\|\sum_{n\in F\cap A_{k_j}}S_ny_j\Big\|\\
&\le \sum_{j=1}^{l} \frac{1}{l2^l}+ \sum_{j=l+1}^{\infty} \frac{1}{j2^j}\le \frac{1}{2^{l-1}}\rightarrow 0.
\end{align*}
We conclude that the vector $x$ is well-defined. To show that $x$ is $\A$-frequently hypercyclic, it is sufficient to prove that there exists a sequence $C_l$ tending to $0$ such that for any $l\ge 1$, any $n\in A_{k_l}$, \[\|T^nx-y_l\|\le C_l.\] Let $l\ge 1$ and $n\in A_{k_l}$. We decompose  $T^nx-y_l$ as follows:
\[ T^nx-y_l=\sum_{\substack{i\in A\\ i<n}}T^nS_iz_i+ \sum_{\substack{i\in A\\ i>n}}T^nS_iz_i+T^nS_ny_l-y_l.\]
We already know by iv. that $\|T^nS_ny_l-y_l\|\le \frac{1}{2^l}$.
Let $m> n$. We also have
\[\sum_{\substack{i\in A\\ n<i\le m}}T^nS_iz_i=\sum_{j=1}^{\infty}\sum_{\substack{i\in A_{k_j}\\ n<i\le m}}T^nS_iy_j=\sum_{j=1}^{\infty}\sum_{i\in F_{k_j}}T^nS_iy_j\]
where $F_{k_j}=A_{k_j}\cap\mathopen]n,m]$. We know by ii. that for any $j\ge 1$, \[\Big\|\sum_{i\in F_{k_j}}T^nS_i y_j\Big\|\le \frac{1}{2^j}\]
and by iii. that for any $j<l$,
\[{\Big\|\sum_{i\in F_{k_j}}T^nS_i y_j\Big\|\le \frac{1}{l2^l}}.\]
We deduce that
\[\Big\|\sum_{\substack{i\in A\\ n<i\le m}}T^nS_iz_i\Big\|\le \frac{1}{2^{l-2}}\quad\text{and thus}\quad \Big\|\sum_{\substack{i\in A\\ i>n}}T^nS_iz_i\Big\|\le \frac{1}{2^{l-2}}.\]
Similarly, we get, with $F_{k_j}=A_{k_j}\cap[0,n[$,
\begin{align*}
\Big\|\sum_{\substack{i\in A\\ i<n}}T^nS_iz_i\Big\|=\Big\|\sum_{j=1}^{\infty}\sum_{i\in F_{k_j}}T^nS_i y_j\Big\|\le \frac{1}{2^{l-2}}.
\end{align*}
We conclude that for any $l\ge 1$, any $n\in A_{k_l}$, \[\|T^nx-y_l\|\le \frac{1}{2^{l-2}}+\frac{1}{2^{l-2}}+\frac{1}{2^l},\]
which concludes the proof.
\end{proof}

We now compare the $\A$-Frequent Hypercyclicity Criterion with the Hypercyclicity Criterion and the Frequent Hypercyclicity Criterion.

\begin{theorem}
Let $X$ be a separable $F$-space and $T\in \mathcal{L}(X)$.
If $T$ satisfies the Hypercyclicity Criterion, then $T$ satisfies
the $\A_\infty$-Frequent Hypercyclicity Criterion.
\end{theorem}

\begin{proof}
Suppose that $T$ satisfies the Hypercyclicity Criterion for $(n_k)$, $(S_{n_k})$, $X_0$ and $Y_0$. Since $X_0$ is dense, we can suppose without loss of generality that $S_{n_k}(Y_0)\subset X_0$ for any $k\ge 1$. Let $(y_j)$ be a dense sequence in $Y_0$. There exists a subsequence $(m_k)\subset (n_k)$ such that for any $j,k< l$, we have
\[\|S_{m_l}y_j\|\le \frac{1}{l^2},\ 
\|T^{m_{k}}S_{m_l}y_{j}\|\le \frac{1}{l^2},\quad \|T^{m_{l}}S_{m_k}y_{j}\|\le \frac{1}{l^2}\quad\text{and}\quad \|T^{m_l}S_{m_l}y_{j}-y_j\|\le \frac{1}{l^2}.
\] 
Therefore, one has that $T$ satisfies the $\A$-Frequent Hypercyclicity Criterion for $A_k=\{m_{p_k^j} : j\in \N\}$, $k\in\N$, where $p_1=2<p_2=3<\dots $ is the increasing enumeration of the prime numbers, and $Y_0=\{y_j : j\in\N \}$.
\end{proof}

\begin{theorem}\label{9fhc ahc}
Let $X$ be a separable $F$-space, $T\in \mathcal{L}(X)$ and $\A$ a hypercyclicity set.
If $T$ satisfies the Frequent Hypercyclicity Criterion, then $T$ satisfies the $\A$-Frequent Hypercyclicity Criterion for any sequence $(A_k)_{k\ge 1}\in \A^\N$ of disjoint sets such that for any $j\in A_k$, any $j'\in A_{k'}$, $j\ne j'$, we have $|j'-j|\ge \max\{k,k'\}$.
\end{theorem}

\begin{proof}
Suppose that $T$ satisfies the Frequent Hypercyclicity Criterion for $Y_0$ and $(S_n)$. Let $\varepsilon>0$ and $y\in Y_0$. We then know that there exists $N\ge 1$ such that for any finite set $F\subset [N,\infty[$, we have
\begin{itemize}
\item $\displaystyle{\Big\|\sum_{n\in F}S_ny\Big\|\le \varepsilon}$,
\item for any $k\ge 0$, $\displaystyle{\Big\|\sum_{n\in F\cap[0,k]}T^{k}S_{k-n}y\Big\|\le \varepsilon}$,
\item for any $k\ge 0$, $\displaystyle{\Big\|\sum_{n\in F}T^{k}S_{k+n}y\Big\|\le \varepsilon}$,
\item for any $n\ge N$, $\|T^nS_ny-y\|\le \varepsilon$.
\end{itemize}
In particular, we deduce that for any $k\ge 0$, any finite subset $F\subset \Z_+$, if $d(F,k)\ge N$, then
\[\Big\|\sum_{n\in F}T^kS_ny\Big\|\le\Big\|\sum_{n\in F\cap[0,k[}T^{k}S_{k-(k-n)}y\Big\|+\Big\|\sum_{n\in F\cap]k,\infty[}T^kS_{k+(n-k)}y\Big\|\le 2\varepsilon.\] Let $(A_k)\subset \A$ be a sequence of disjoint sets such that for any $j\in A_k$, any $j'\in A_{k'}$, $j\ne j'$, we have $|j'-j|\ge \max\{k,k'\}$. We deduce that the operator $T$ satisfies the assertions~$2.$ and~$3.$ of the $\A$-Frequent Hypercyclicity Criterion for the sequence $(A_k)$. It is obvious that $T$ satisfies the assertion~$1.$ of the $\A$-Frequent Hypercyclicity Criterion and since for any $N\ge 1$, there exists $k_0$ such that for any $k\ge k_0$, $A_k\cap [0,N]=\emptyset$, we conclude that $T$ also satisfies the assertion~$4.$ of the $\A$-Frequent Hypercyclicity Criterion for the sequence $(A_k)$.
\end{proof}

In particular, we deduce the following result from Theorem~\ref{9fhc ahc}.

\begin{corollary}
Let $X$ be a separable $F$-space and $T\in \mathcal{L}(X)$.
If $T$ satisfies the Frequent Hypercyclicity Criterion, then $T$ is $\A$-frequently hypercyclic for any hypercyclicity set $\A$.
\end{corollary}

\subsection{Characterization of $\A$-frequently hypercyclic weighted shifts}\label{SEC Shift}

Weighted shifts are one of the most important family of operators in linear dynamics. The goal of this section consists in showing how we can apply the $\A$- Hypercyclicity Criterion in order to characterize the $\A$-frequently hypercyclic weighted shifts on $\ell^p(\Z_+)$ ($1\le p<\infty$) or on $c_0(\Z_+)$.

\begin{theorem}\label{characshift}
Let $\A$ be a hypercyclicity set and $B_w$ a weighted shift on $X$ where $X=\ell^p(\Z_+)$ $(1\le p<\infty)$ or $c_0(\Z_+)$.
The following assertions are equivalent:
\begin{enumerate}
\item[1)] $B_w$ is $\A$-frequently hypercyclic,
\item[2)] $B_w$ satisfies the $\A$-Frequent Hypercyclicity Criterion,
\item[3)] there is a sequence $(A_k)_{k\ge 1}\subset \A$ of disjoints sets such that
\begin{enumerate}
\item[i.] for any $j\in A_k$, any $j'\in A_{k'}$, $j\ne j'$, we have $|j'-j|\ge \max\{k,k'\}$.
\item[ii.] for any $k'\ge 0$, any ${k> k'}$,
\[
\sum_{n\in A_k+k'}\frac{e_{n}}{\prod_{\nu=1}^nw_{\nu}}\in X\quad\text{and}\quad \sum_{n\in A_k+k'}\frac{e_n}{\prod_{\nu=1}^nw_{\nu}}\xrightarrow[]{k\to \infty}0;
\]
\item[iii.] there exists a family $(C_{k,l})_{k,l\ge 1}$ such that for any $k'\ge 0$, any ${k> k'}$, any $l\ge 1$,
\[
\sup_{j\in A_l}\Big\|\sum_{n\in A_k-j}\frac{e_{n+k'}}{\prod_{\nu=1}^nw_{\nu+k'}}\Big\|\le C_{k,l}
\]
and such that $\sup_l C_{k,l}$ converges to $0$ when $k\rightarrow \infty$ and, for any $k\ge 0$, $C_{k,l}$ converges to $0$ when $l\rightarrow \infty$.
\end{enumerate}
\end{enumerate}
\end{theorem}

\begin{remark}
In the statement of this theorem and the rest of this section, for any $A\subset \Z_+$, any $j\in\Z_+$, we denote by $\sum_{n\in A-j}$ the series $\sum_{n\in(A-j)\cap \N}$ and by
 $\sum_{n\in j-A}$ the series $\sum_{n\in(j-A)\cap \N}$. We also suppose by convention that $\prod_{\nu=1}^{0}w_\nu=1$.
\end{remark}

\begin{proof}
$3)\Rightarrow 2)$. Let $S((x_n)_{n\ge 0})=(0,\frac{x_0}{w_1},\frac{x_1}{w_2},\dots)$. We show that $B_w$ satisfies the $\A$-Frequent Hypercyclicity Criterion for the set $Y_0=\text{span}\{e_k:k\ge 0\}$, the maps $(S^n)_{n\ge 0}$ and the sets $A_k$ given by the assertion $3)$.
We notice that for any $x\in Y_0$, any $n\ge 1$, any finite set $F\subset \Z_+$, we have
\begin{equation}
\sum_{i\in F\backslash\{n\}}B_w^nS^ix=\sum_{\substack{i\in F\\ i<n}}B_w^{n-i}x+\sum_{\substack{i\in F\\ i>n}}S^{i-n}x
=\sum_{i\in n-F}B_w^{i}x+\sum_{i\in F-n}S^{i}x.
\label{eq B}
\end{equation}

We prove that $B_w$ satisfies each assumption of the $\A$-Frequent Hypercyclicity Criterion for any vector $e_{k'}$:
\begin{enumerate}
\item Let $k'\ge 0$. For any $k> k'$, we deduce from ii. that
\[\sum_{n\in A_k}S^ne_{k'}=\sum_{n\in A_k}\frac{e_{k'+n}}{\prod_{\nu=1}^nw_{\nu+k'}}=\prod_{\nu=1}^{k'}w_{\nu}\cdot\sum_{n\in A_k+k'}\frac{e_{n}}{\prod_{\nu=1}^nw_{\nu}}\in X.\]
In particular, each series $\sum_{n\in A_k}S^ne_{k'}$ converges unconditionally and since $\sum_{n\in A_k}S^ne_{k'}$ tends to $0$ as $k\to \infty$ by ii., we deduce that $\sum_{n\in A_k}S^ne_{k'}$ converges unconditionally, uniformly in $k> k'$.
\item Let $k'\ge 0$, $k_0\ge 0$ and $\varepsilon>0$. For any $k> k'$, we deduce from~i. that for any $j\in \bigcup_l A_l$, \[\sum_{n\in j-A_k}B_w^ne_{k'}=0\] and if $j\in A_l$, we deduce from~iii. that \[\Big\|\sum_{n\in A_k-j}S^ne_{k'}\Big\|\le C_{k,l}.\]
Since $\sup_{l}C_{k,l}\rightarrow 0$ when $k\to \infty$, there exists $k>\max\{k',k_0\}$ such that we have for any $j\in \bigcup_l A_l$, $\|\sum_{n\in A_k-j}S^ne_{k'}\|\le \varepsilon$. On the other hand, since $C_{k,l}$ converges to $0$ when $l\rightarrow 0$, for any $\delta>0$, there exists $l_0\ge 1$ such that for any $l\ge l_0$, we have $C_{k,l}\le \delta$. We  conclude by using~$\eqref{eq B}$.

\item Obvious.\\
\end{enumerate}

\noindent $2)\Rightarrow 1)$ follows from Theorem~\ref{Ahypc}.\\

\noindent $1)\Rightarrow 3)$. Let $x$ be an $\A$-frequently hypercyclic vector for $B_w$. For any $k\ge 1$, any $C_k>0$, any $\varepsilon_k>0$, there exists $A_k \in \A$ such that for any $n\in A_k$,
\[\Big\|B^n_wx-(C_k+\varepsilon_k)\sum_{k'=0}^{k}e_{k'}\Big\|<\varepsilon_k.\]
We will fix $C_k$ and $\varepsilon_k$ later but we already suppose that
$(C_k)$ is an increasing sequence tending to $\infty$ and $(\varepsilon_k)$ is a decreasing sequence tending to $0$. We first notice that for any $n\in A_k$, any $k'> k$, we have
\begin{equation}
\Big|\prod_{\nu=1}^nw_{\nu+k'}x_{n+k'}\Big|<\varepsilon_k
\label{eq: fhc0}
\end{equation}
and for any $0 \le k'\le k$, we have
\[\Big|\prod_{\nu=1}^nw_{\nu+k'}x_{n+k'}- (C_k+\varepsilon_k)\Big|<\varepsilon_k\]
and thus
\begin{equation}
C_k<\Big|\prod_{\nu=1}^nw_{\nu+k'}x_{n+k'}\Big|<C_k+2\varepsilon_k.
\label{eq: fhc}
\end{equation}
We show that the sequence $(A_k)_{k\ge 1}$ satisfies i., ii. and iii. for a good choice of $(C_k)$ and $(\varepsilon_k)$:
\begin{enumerate}
\item  We show that the sets $(A_k)$ are disjoints and that for any $j\in A_k$, any $j'\in A_{k'}$, $j\ne j'$, we have $|j'-j|> \max\{k,k'\}$.

Let $1\le k'< k$ and $j\in A_k$. We have to show that for any $0\le n\le k$, we have $j+n\notin A_{k'}$ and $j-n\notin A_{k'}$. Let $0\le n\le k$.
We know by \eqref{eq: fhc} that $\Big|\prod_{\nu=1}^jw_{\nu+k}x_{j+k}\Big|>C_k$. We deduce that
\begin{equation}
\Big|\prod_{\nu=1}^{j+n}w_{\nu+k-n}x_{j+k}\Big|>C_k\min_{0\le n'\le k}\min_{\nu\le k}|w_{\nu}|^{n'} \label{eq3}
\end{equation}
and
\begin{equation}
\Big|\prod_{\nu=1}^{j-n}w_{\nu+k+n}x_{j+k}\Big|>\frac{C_k}{\max_{0\le n'\le k}\max_{k\le \nu\le 2k}|w_{\nu}|^{n'}}\label{eq4}.
\end{equation}
On the other hand, we also know by \eqref{eq: fhc0} and \eqref{eq: fhc} that for any $m\in A_{k'}$, any $l\ge 0$,
\begin{equation}
\Big|\prod_{\nu=1}^mw_{\nu+l}x_{m+l}\Big|<C_{k'}+2\varepsilon_{k'}.\label{eq5}
\end{equation}
In particular, if $j+n\in A_{k'}$, we would have, for $l=k-n$,
\begin{equation}\Big|\prod_{\nu=1}^{j+n}w_{\nu+k-n}x_{j+k}\Big|<C_{k'}+2\varepsilon_{k'}\le C_{k-1}+2\varepsilon_{1}.\label{eq6}\end{equation}
We then deduce from \eqref{eq3} and \eqref{eq6} that if we suppose
\begin{equation}
C_k>\frac{C_{k-1}+2\varepsilon_{1}}{\min_{0\le n'\le k}\min_{\nu\le k}|w_{\nu}|^{n'}},
\label{cond 1}
\end{equation}
then for any $k'< k$, any $0\le n\le k$, we have $j+n\notin A_{k'}$.
On the other hand, if we suppose that $j-n\in A_{k'}$, then we would have by \eqref{eq5} for $l=k+n$,
\begin{equation}
\Big|\prod_{\nu=1}^{j-n}w_{\nu+k+n}x_{j+k}\Big|<C_{k'}+2\varepsilon_{k'} \le C_{k-1}+2\varepsilon_{1}.
\label{eq7}
\end{equation}
Hence, if we suppose
\begin{equation}
C_k>(C_{k-1}+2\varepsilon_{1})\max_{0\le n'<k}\max_{k\le \nu\le 2k}|w_{\nu}|^{n'},
\label{cond 2}
\end{equation}
then, by \eqref{eq4} and \eqref{eq7}, we deduce that for any $k'< k$, any $0\le n\le k$, we have $j-n\notin A_{k'}$ and thus $d(A_{k'},A_k)>k$.

If $k=k'$, it suffices to show that for any $j\in A_k$, any $1\le n\le k$, we have $j-n\notin A_k$. By \eqref{eq: fhc0}, we know that for any $m\in A_k$, any $l>k$,
\begin{equation*}
\Big|\prod_{\nu=1}^mw_{\nu+l}x_{m+l}\Big|<\varepsilon_{k}.
\end{equation*}
If $j-n\in A_{k}$, we would thus have, for $l=k+n$,
\begin{equation}\Big|\prod_{\nu=1}^{j-n}w_{\nu+k+n}x_{j+k}\Big|<\varepsilon_{k}\le \varepsilon_{1}.\label{eq6b}\end{equation}
We deduce from \eqref{eq4} and \eqref{eq6b} that if
\begin{equation}
C_k>\varepsilon_{1}\max_{0\le n'\le k}\max_{k\le \nu\le 2k}|w_{\nu}|^{n'}
\label{cond 3}
\end{equation}
then  $j-n\notin A_k$ for any $1\le n\le k$.

The property i. is thus satisfied if we choose $(C_k)$ such that Conditions \eqref{cond 1}, \eqref{cond 2} and \eqref{cond 3} are satisfied. For the sequence $(\varepsilon_k)$, we can consider any decreasing sequence tending to $0$.

\item Let $k'\ge 0$. We know by \eqref{eq: fhc} that for any $k> k'$, any $n\in A_k$,
\[\frac{1}{\prod_{\nu=1}^n|w_{\nu+k'}|}<\frac{|x_{n+k'}|}{C_{k}}.\]
Hence, we have
\begin{align*}
\Big\|\sum_{n\in A_k+k'}\frac{e_n}{\prod_{\nu=1}^nw_{\nu}}\Big\|&=\Big\|\sum_{n\in A_k}\frac{e_{n+k'}}{\prod_{\nu=1}^{n+k'}w_{\nu}}\Big\|\\
&=\frac{1}{\prod_{\nu=1}^{k'}|w_{\nu}|}
\Big\|\sum_{n\in A_k}\frac{e_{n+k'}}{\prod_{\nu=1}^{n}|w_{\nu+k'}|}\Big\|\\
&\le \frac{\|x\|}{C_k\prod_{\nu=1}^{k'}|w_{\nu}|}<\infty
\end{align*}
and since $C_k\to \infty$, we deduce that
\[\sum_{n\in A_k+k'}\frac{e_n}{\prod_{\nu=1}^nw_{\nu}}\xrightarrow[]{k\to \infty}0.\]
\item We show that for any $k'\ge 0$, any $k> k'$, any $l\ge 1$, any $j\in A_l$, we have
\[\Big\|\sum_{n\in A_k-j} \frac{C_k}{\prod_{\nu=1}^{n}w_{\nu+k'}}e_{n+k'}\Big\|<\varepsilon_l.\]
Let  $k>k'\ge 0$, $l\ge 1$ and $j\in A_l$. We have, by definition of $A_l$,
\begin{align*}
\varepsilon_l&>\Big\|B_w^jx-(C_l+\varepsilon_l)\sum_{l'=0}^le_{l'}\Big\|\ge \Big\|\sum_{n=l+1}^{\infty}\Big(\prod_{\nu=1}^jw_{\nu+n}\Big)x_{j+n}e_{n}\Big\|\\
&\ge \Big\|\sum_{n=l+1}^{\infty}\Big(\prod_{\nu=1}^jw_{\nu+n+k'}\Big)x_{j+n+k'}e_{n+k'}\Big\|\\
&\ge \Big\|\sum_{n\in A_k-j}\Big(\prod_{\nu=1}^jw_{\nu+n+k'}\Big)x_{j+n+k'}e_{n+k'}\Big\| \quad \text{by 1.}\\
&= \Big\|\sum_{\substack{n\in A_k\\ n>j}}\Big(\prod_{\nu=1}^jw_{\nu+n-j+k'}\Big)x_{n+k'}e_{n-j+k'}\Big\|\\
&= \Big\|\sum_{\substack{n\in A_k\\ n>j}} \Big(\frac{\prod_{\nu=1}^nw_{\nu+k'}}{\prod_{\nu=1}^{n-j}w_{\nu+k'}}\Big) x_{n+k'}e_{n-j+k'}\Big\|\\
&\ge \Big\|\sum_{\substack{n\in A_k\\ n>j}} \frac{C_k}{\prod_{\nu=1}^{n-j}w_{\nu+k'}}e_{n-j+k'}\Big\| \quad\text{by \eqref{eq: fhc}}\\
&=\Big\|\sum_{n\in A_k-j} \frac{C_k}{\prod_{\nu=1}^{n}w_{\nu+k'}}e_{n+k'}\Big\|.
\end{align*}
We deduce that
\[\Big\|\sum_{n\in A_k-j}\frac{e_{n+k'}}{\prod_{\nu=1}^nw_{\nu+k'}}\Big\|\le \frac{\varepsilon_l}{C_k}\]
and since
\[\sup_l \frac{\varepsilon_l}{C_k}\le \frac{\varepsilon_1}{C_k}\xrightarrow[k\rightarrow \infty]{} 0 \quad\text{and}\quad
\frac{\varepsilon_l}{C_k}\xrightarrow[l\rightarrow \infty]{}0,\]
we obtain the desired result.
\end{enumerate}
\end{proof}

An important result of Bayart and Ruzsa \cite{Bayart2} about frequently hypercyclic weighted shifts on $\ell^p$  is that a weighted shift $B_w$ on $\ell^p$ is frequently hypercyclic if and only if $B_w$ is chaotic and thus if and only if
\begin{equation}
\sum_{n=1}^{\infty}\frac{1}{\prod_{\nu=1}^{n}|w_{\nu}|^p}<\infty.
\label{9caraclp}
\end{equation}
However, we know that this equivalence is false for weighted shifts on $c_0$. Indeed, Bayart and Grivaux \cite{2Bayart1} have exhibited a frequently hypercyclic weighted shift on $c_0$ that is neither chaotic nor mixing. The characterization of frequently hypercyclic weighted shifts that we obtain in term of weights is not satisfactory in the case of the spaces $\ell^p$ if we compare it with the characterization obtained by Bayart and Ruzsa. However, in the case of frequently hypercyclic weighted shifts on $c_0$, the obtained characterization is similar to the characterization given in \cite{Bayart2}.

Thanks to the counterexample of Bayart and Grivaux, we also know that the Frequent Hypercyclicity Criterion does not characterize frequently hypercyclic operators, because if $T$ satisfies the Frequent Hypercyclicity Criterion, then $T$ is mixing and chaotic. However, the characterization given by Theorem~\ref{characshift} tells us that each frequently hypercyclic weighted shift on $c_0$ satisfies the $\A$-Frequent Hypercyclicity Criterion when $\A$ is the family of positive lower density sets. A direct consequence is the existence of operators that are neither mixing nor chaotic and that satisfy the $\A$-Frequent Hypercyclicity Criterion when $\A$ is the family of positive lower density sets. One can thus wonder if the $\A$-Frequent Hypercyclicity Criterion characterizes the frequently hypercyclic operators when $\A$ is the set of positive lower density sets.

\section{Banach densities and reiterative hypercyclicity}\label{dens}

The purpose of this section is to analyze which kind of properties of density can have the sets $N(x, U)$ and classify the hypercyclic operators accordingly to these properties. We first remark that there does not exist $\A$-frequently hypercyclic operators if $\A$ is the family of sets with positive lower Banach density or if $\A$ is the family of sets with upper Banach density equal to~$1$.

\begin{proposition}\label{prop.syndetic}
Let $X\ne\{0\}$ be a $F$-space. If $\A$ is the family of sets with positive lower Banach density, then $X$ does not support an $\A$-frequently hypercyclic operator.
\end{proposition}

\begin{proof}
 Assume towards a contradiction that there exists an $\A$-frequently hypercyclic operator $T$ on $X$. Let $x\in X$ be an $\A$-frequently hypercyclic vector for $T$.

Take $U, V$ non-empty open sets such that $x\in U$, $0\in V$ and $U\cap V=\emptyset$. If we denote the maximum gap of $N(x, U)$ as $m$, then by continuity there exists $W$ a neighbourhood of zero such that $T^j(W)\subset V, j=0,1,\dots,m$. Let $n$ such that $T^nx\in W$. We deduce that $T^{n+j}x\in V$ for any $0\le j \le m$ and therefore $\{n, n+1, \dots, n+m\}\notin N(x, U)$ which is a contradiction since this implies that there are gaps in $N(x, U)$ with length greater than~$m$.
\end{proof}

\begin{proposition}\label{prop.thick}
 Let $X\ne\{0\}$ be a $F$-space. If $\A$ is the family of sets with upper Banach density equals to~$1$, then $X$ does not support an $\A$-frequently hypercyclic operator.
\end{proposition}
\begin{proof}
 Assume towards a contradiction that there exists an $\A$-frequently hypercyclic operator $T$ on $X$. Let $x\in X$ be an $\A$-frequently hypercyclic operator. This implies that for every non-empty open set $U$, there exists $n$ such that $n, n+1\in N(x, U)$ and thus for every non-empty open set $U$, we have
 \begin{equation}
 \label{eq.TU}
 T(U)\cap U \neq\emptyset.
 \end{equation}
Let $z\ne y\in X$ such that $Tz=y$. Since $X$ is metrizable, there exists open neighborhoods $V_y$, $V_z$ of $y$ and $z$ respectively, such that $V_y\cap V_z=\emptyset$. On the other hand, by continuity of $T$ there exists an open neighborhood of $z$ denoted $\tilde{V_z}$, such that $T(\tilde{V_z})\subset V_y$. Let $\hat{V_z}=V_z\cap \tilde{V_z}$. We get that $\hat{V_z}$ is a non-empty open set and since $T(\hat{V_z})\subset V_y$, we conclude that $T(\hat{V_z})\cap \hat{V_z}=\emptyset$. This is a contradiction with condition \eqref{eq.TU}.
\end{proof}

We deduce from Proposition \ref{prop.syndetic}, Proposition \ref{prop.thick} and \eqref{eq.density} that there are only three possibilities:

\begin{enumerate}
\item $T$ is frequently hypercyclic i.e. $\underline{d}(N(x,U))>0$;
\item $T$ is $\mathfrak{U}$-frequently hypercyclic i.e. $\overline{d}(N(x,U))>0$;
\item  $T$ is reiteratively hypercyclic i.e. $\overline{Bd}(N(x,U))>0$.
\end{enumerate}

Obviously, for any operator $T$, we have
\[\text{frequently hypercyclic} \Rightarrow \text{$\mathfrak{U}$-frequently hypercyclic} \Rightarrow \text{reiteratively hypercyclic}.\]
In the case of weighted shifts on $\ell^p$, we have even an equivalence between frequent hypercyclicity and $\mathfrak{U}$-frequent hypercyclicity~\cite{Bayart2}. Nevertheless, this equivalence is false in general and in particular for weighted shifts on $c_0$~\cite{Bayart2}. Thanks to weighted shifts on $c_0$, we can also show that there is not, in general, an equivalence between reiterative hypercyclicity  and $\mathfrak{U}$-frequent hypercyclicity.

\begin{theorem}\label{counterexc0}
There exists some reiteratively hypercyclic weighted shift on $c_0(\Z_+)$ that is not $\mathfrak{U}$-frequently hypercyclic.
\end{theorem}
\begin{proof}
Let $S:=\bigcup_{j,l\ge 1}]l 10^j-j,l 10^j+j[$.
We consider the weighted sequence $w$ given by
\[w_k=\left\{
\begin{array}{cl} \displaystyle{2} & \quad\mbox{if $k\in S$}\\
 {\displaystyle \prod_{\nu=1}^{k-1}w_{\nu}^{-1}} & \quad\mbox{if $k\in (S+1)\backslash S$}\\
  1 & \quad\mbox{otherwise}.
\end{array}\right.\]
In particular, we deduce from the definition that $\prod_{\nu=1}^nw_{\nu}=1$ if and only if $n\notin S$, and $\prod_{\nu=1}^nw_{\nu}\ge 2^{j}$ if and only if $]n-j,n]\subset S$.

We first show that $B_w$ is reiteratively hypercyclic by using Theorem~\ref{characshift}.
To this end, we have to construct a sequence $(A_k)_{k\ge 1}$ of disjoints sets with positive upper Banach density such that
\begin{enumerate}
\item[i.] for any $j\in A_k$, any $j'\in A_{k'}$, $j\ne j'$, we have $|j'-j|\ge \max\{k,k'\}$.
\item[ii.] for any $k'\ge 0$, any ${k> k'}$,
\[
\prod_{\nu=1}^nw_{\nu}\xrightarrow{n\in A_k+k'} \infty\quad \text{and}
\quad \sup_{n\in A_k+k'}\frac{1}{\prod_{\nu=1}^n w_\nu}\xrightarrow[]{k\to \infty}0;
\]
\item[iii.] there exists a family $(C_{k,l})_{k,l\ge 1}$ such that for any $k'\ge 0$, any ${k> k'}$, any $l\ge 1$,
\[
\sup_{j\in A_l}\sup_{n\in A_k-j}\frac{1}{\prod_{\nu=1}^nw_{\nu+k'}}\le C_{k,l}
\]
and such that $\sup_l C_{k,l}$ converges to $0$ when $k\rightarrow \infty$ and, for any $k\ge 0$, $C_{k,l}$ converges to $0$ when $l\rightarrow \infty$.
\end{enumerate}

Let $\phi: \mathbb{N}\to \mathbb{N}$ such that for any $k\ge 1$, we have
${\#\{j\ge 1:\phi(j)=k\}=\infty}$. We construct a sequence of sets $(F_{j})_{j\ge 1}$ such that
if $A_k:=\bigcup_{\phi(j)=k} F_j$, then the sequence $(A_k)_{k\ge 1}$ satisfies the desired properties.

Let $F_{0}=\{0\}$. If we have already constructed $F_1,\cdots,F_j$ and $\phi(j+1)=k$, then
we let $F_{j+1}=\{10^{j_0}+10^{2k}l:l\in [0,l_0[\}$ where $j_0$ and $l_0$ are positive integers satisfying
\begin{enumerate}
\item[1)] $l_0\ge j+1$;
\item[2)] $10^{j_0}\ge k+\max_{1\le n \le j}\phi(n)+\max\Big(\bigcup_{n=0}^{j}F_n\Big)$;
\item[3)] $j_0\ge j+1$ and $j_0-k>10^{2k}l_0$;
\item[4)] $j_0> \max\Big(\bigcup_{n=0}^{j}F_n\Big)+\max_{1\le n \le j}\phi(n)+2k$.
\end{enumerate}

Let $A_k:=\bigcup_{\phi(j)=k} F_j$.
We first remark that for any real number $s\ge 1$, any $k\ge 1$,
\begin{align*}
\alpha_k^s:=\limsup_{l\to \infty}|A_k\cap[l+1,l+s]|&\ge \limsup_{j\to \infty,\phi(j)=k}|F_j\cap[\min F_j+1,\min F_j+s]|\\
&\ge \limsup_{j\to \infty,\phi(j)=k} \min\left\{\left\lfloor\frac{s}{10^{2k}}\right\rfloor,j\right\} \quad \text{by 1)}\\
&=\left\lfloor\frac{s}{10^{2k}}\right\rfloor.
\end{align*}
We deduce that for any $k\ge 1$
\[\overline{Bd}(A_k):=\lim_{s\to \infty}\frac{\alpha_k^s}{s}\ge\lim_{s\to \infty}\frac{\left\lfloor\frac{s}{10^{2k}}\right\rfloor}{s}=\frac{1}{10^{2k}}>0.\]

Moreover, the sets $A_k$ are disjoint since it follows from~2) that the sets $F_j$ are disjoint.
In fact, 2) implies that for any $j\in A_k$, any $j'\in A_{k'}$, $j\ne j'$, we have $|j'-j|\ge \max\{k,k'\}$. Indeed if $j\in A_k$, there exists $n\ge 1$ such that $j\in F_n$ and $\phi(n)=k$, and if $j'\in A_{k'}$, there exists $n'\ge 1$ such that $j'\in F_{n'}$ and $\phi(n')=k'$. Therefore, if $n=n'$, we have $k=k'$ and $|j-j'|$ is a non-zero multiple of $10^{2k}\ge k$. On the other hand, if $n\ne n'$,  we can assume without loss of generality that $n'< n$ and we deduce from 2) that
\[|j-j'|\ge k+\max_{1\le m< n}\phi(m)\ge k+\phi(n')=k+k'\ge \max\{k,k'\}.\]

It remains to prove that the sets $A_k$ satisfy Condition ii. and iii. of Theorem~\ref{characshift}.
Condition ii. is satisfied because for any $k\ge 1$, any $j\ge 1$ with $\phi(j)=k$, it follows from~3) that
$F_j\subset [10^{j_0},10^{j_0}+j_0-k[$ for some $j_0\ge j$. Therefore, for any $k'<k$, we have $F_j+k'\subset [10^{j_0},10^{j_0}+j_0[$ and thus, by definition of $w$, for any $n\in F_j+k'$, we have
\[\prod_{\nu=1}^nw_{\nu}\ge 2^{j_0}\ge 2^j.\]
Since each $F_j$ is finite, we then get, by definition of $A_k$, for any $k'<k$
\[\prod_{\nu=1}^nw_{\nu}\xrightarrow{n\in A_k+k'} \infty\]
and since $\min( \phi^{-1}(\{k\}))\to \infty$ as $k\to \infty$, we get
\[\sup_{n\in A_k+k'}\frac{1}{\prod_{\nu=1}^nw_{\nu}}\xrightarrow[]{k\to \infty}0.\]

Finally, for any $k'\ge 0$, any ${k> k'}$, any $j\ge 1$ such that $\phi(j)=k$, if $n=m-m'\ge 1$ with $m\in F_j$ and $m'\in F_{j'}$, we have two possibilities either $j=j'$ or $j>j'$.
If $j=j'$, we deduce from the definition of $F_j$ that $n=l 10^{2k}$ for some $l\ge 1$ and thus
\[\prod_{\nu=1}^nw_{\nu+k'}= \frac{\prod_{\nu=1}^{n+k'}w_{\nu}}{\prod_{\nu=1}^{k'}w_{\nu}}\ge \frac{2^{2k}}{2^{k'}}\ge 2^{k}.\]
On the other hand, if $j>j'$, we deduce from 4) that \[n\in [10^{j_0}+2k+\phi(j')-j_0,10^{j_0}+j_0-k[\] for some $j_0\ge 1$ and thus
$n+k'\in [10^{j_0}+2k+\phi(j')-j_0,10^{j_0}+j_0[$. We conclude that
\[\prod_{\nu=1}^nw_{\nu+k'}= \frac{\prod_{\nu=1}^{n+k'}w_{\nu}}{\prod_{\nu=1}^{k'}w_{\nu}}\ge \frac{2^{2k+\phi(j')}}{2^{k'}}\ge 2^{k+\phi(j')}.\]

Let  $C_{k,k}:=2^{-k}$ and $C_{k,l}:=2^{-(k+l)}$ if $k\ne l$.
We deduce that for any $k'\ge 0$, any ${k> k'}$, any $l\ge 1$, we have
\[\sup_{m'\in A_l}\max_{n\in A_k-m'}\frac{1}{\prod_{\nu=1}^nw_{\nu+k'}}\le C_{k,l}\]
and we remark that $\sup_l C_{k,l}=2^{-k}$ converges to $0$ when $k\rightarrow \infty$ and, for any $k\ge 0$, $C_{k,l}$ converges to $0$ when $l\rightarrow \infty$. Condition iii. of Theorem~\ref{characshift} is thus satisfied and we conclude that $B_w$ is reiteratively hypercyclic.\\

We now show that $B_w$ is not $\mathfrak{U}$-frequently hypercyclic. Assume that $B_w$ is $\mathfrak{U}$-frequently hypercyclic. We then deduce from Theorem~\ref{characshift} that there exists a set $A$ with positive upper density such that \[\sum_{n\in A}\frac{e_{n}}{\prod_{\nu=1}^nw_{\nu}}\in c_0.\]
In other words, we have $\overline{d}(A)>0$ and
\[ \prod_{\nu=1}^nw_{\nu}\xrightarrow{n\in A} \infty.\]
Let $D_j:=\{n\ge 1:\prod_{\nu=1}^n w_{\nu}\ge 2^j\}$. We remark that $\overline{d}(A)\le \overline{d}(D_j)$. In order to prove that  $B_w$ is not $\mathfrak{U}$-frequently hypercyclic, it is thus sufficient to prove that $\overline{d}(D_j)\to 0$ when $j\to \infty$. To this end, we will need the following fact.\\

\noindent\textbf{Fact 1.} \emph{Let $S:=\bigcup_{j,l\ge 1}]l 10^j-j,l 10^j+j[$. Let $k\ge 1$, $l\ge 1$ and $n\ge 0$ such that $10^{n-1}<k\le 10^n$.
If $m\in\{l10^k+\sum_{j=0}^{n}10^j,l10^k-\sum_{j=0}^{n}10^j\}$, then
\begin{equation*}
\text{either } m\notin S \text{ or }m\in ]l_0 10^{j_0}-j_0,l_010^{j_0}+j_0[ \text{ for some $j_0>k$.}
\end{equation*}}

\noindent Thanks to Fact 1, we can show that
\[D_j:=\{n\ge 1:\prod_{\nu=1}^n w_{\nu}\ge 2^j\}\subseteq \bigcup_{k\ge \lceil\frac{j}{30}\rceil}\bigcup_{l\ge 1}\ ]l10^k-31k,l10^k+31k[:=E_j.\]
Let $n\notin E_j$. We want to show that $n\notin D_j$. In other words, we have to show that
\[]n-j,n]\cap S^c\ne \emptyset.\]
We first remark that $]n-j,n]\cap ]l10^k-k,l10^k+k[=\emptyset$ for any $k\ge \lceil\frac{j}{30}\rceil$ and any $l\ge 1$. Indeed, if $]n-j,n]\cap ]l10^k-k,l10^k+k[\ne\emptyset$, then $n\ge l10^k-k$ and $n-j<l10^k+k$. Hence, \[l10^k-31k<n=(n-j)+j<l10^k+k+j\le l10^k+31k\]
and thus $n\in E_j$, which is a contradiction.

Assume that $]n-j,n]\subset S$ and let
\[k_0:=\max\left\{k\ge 1:]n-j,n]\cap \bigcup_{l\ge 1}]l10^k-k,l10^k+k[\ne \emptyset\right\}.\] We deduce from the above reasoning that $k_0< \lceil\frac{j}{30}\rceil$.  We consider $m\in ]n-j,n]$ and $l_0\ge 1$ such that \[m\in ]l_010^{k_0}-k_0,l_010^{k_0}+k_0[.\]  Let $n\ge 0$ such that $10^{n-1}<k_0\le 10^n$, $m_1:=l_010^{k_0}-\sum_{j=0}^{n}10^j$ and $m_2:=l_010^{k_0}+\sum_{j=0}^{n}10^j$. Since $m\in ]n-j,n]$, $m\in [m_1,m_2]$ and
\[m_2-m_1+1\le 3 (10)^{n}<30 k_0\le j,\] we deduce that either $m_1\in ]n-j,n]$ or $m_2\in ]n-j,n]$. Therefore $m_1$ or $m_2$ belongs to $S$ and we deduce from Fact 1 that $m_1$ or $m_2$ belongs to $]l_110^{k_1}-k_1,l_110^{k_1}+k_1[$ for some $k_1>k_0$ and some $l_1 \ge 1$ which is a contradiction with the definition of $k_0$.

If $N\in [10^m,10^{m+1}[$ with $m\ge 1$, we thus deduce that
\begin{align*}
\frac{\#(D_j\cap [1,N])}{N}&\le
\frac{\#(D_j\cap [1,10^{m+1}])}{10^m}\\
&\le \frac{\#(E_j\cap [1,10^{m+1}])}{10^{m}}
\le \frac{\sum_{k=\lceil \frac{j}{30} \rceil}^{m}62k10^{m+1-k}}{10^{m}}\\
&= \frac{620}{81}\left(\frac{(9\lceil \frac{j}{30} \rceil+1)10^{m+1-\lceil \frac{j}{30} \rceil}-9m-10}{10^{m}}\right)\\
&\le 8\left(9\left\lceil \frac{j}{30} \right\rceil+1\right)10^{1-\lceil \frac{j}{30} \rceil}.
\end{align*}
where the third inequality follows from the fact that for any $\lceil \frac{j}{30} \rceil\le k\le m-1$, there is less than $10^{m+1-k}$ intervals of the form $]l10^{k}-31k,l10^{k}+31k[$ with $l\ge 1$ in $E_j\cap [1,10^{m+1}]$, $9$ intervals of the form
$]l10^{m}-31m,l10^{m}+31m[$ and the interval $]10^{m+1}-31(m+1),10^{m+1}]$. We deduce that $\overline{d}(D_j)\to 0$ as $j\to \infty$ and we thus conclude that $B_w$ is not $\mathfrak{U}$-frequently hypercyclic.\\

We finish this proof by giving the proof of Fact 1.\\
\noindent\textbf{Proof of Fact 1.}
We first assume that $m=l10^k+\sum_{j=0}^{n}10^j$. If $m\in S$, then there exists $l_0\ge 1$ and $j_0\ge 1$ such that $m\in  ]l_0 10^{j_0}-j_0,l_010^{j_0}+j_0[$. Therefore, it suffices to prove that
 $m\notin ]l_1 10^{j_1}-j_1,l_110^{j_1}+j_1[$ for any $1\le j_1\le k$, any $l_1\ge 1$.

If $n+1\le j_1\le k$, then $l 10^{k-j_1}10^{j_1} \le m \le (l 10^{k-j_1}+1)10^{j_1}$ and we remark that
\[m\ge l 10^{k}+10^n \ge l 10^{k-j_1}10^{j_1}+j_1\]
and
\[m\le l 10^{k}+10^{j_1}-10^{j_1-1}\le (l 10^{k-j_1}+1)10^{j_1}-j_1.\]
Hence $m\notin ]l_1 10^{j_1}-j_1,l_110^{j_1}+j_1[$ for any $n+1\le j_1\le k$ and any $l_1\ge 1$.
On the other hand, if $1\le j_1\le n$ then
\[(l10^{k-j_1}+\sum_{j=j_1}^{n}10^{j-j_1})10^{j_1}\le m\le (l10^{k-j_1}+\sum_{j=j_1}^{n}10^{j-j_1}+1)10^{j_1}.\]
However,
\[m\ge l10^k+\sum_{j=j_1}^{n}10^{j}+10^{j_1-1}\ge (l10^{k-j_1}+\sum_{j=j_1}^{n}10^{j-j_1})10^{j_1}+j_1\]
and
\[m\le l10^k+\sum_{j=j_1}^{n}10^{j}+10^{j_1}-10^{j_1-1}\le(l10^{k-j_1}+\sum_{j=j_1}^{n}10^{j-j_1}+1)10^{j_1}-j_1.\]
So we conclude that $m\notin ]l_1 10^{j_1}-j_1,l_110^{j_1}+j_1[$ for any $1\le j_1\le k$ and any $l_1\ge 1$.\\

If we assume that $m=l10^k-\sum_{j=0}^{n}10^j$ and $m\in S$, we can show similarly that  $m\notin ]l_1 10^{j_1}-j_1,l_110^{j_1}+j_1[$ for any $1\le j_1\le k$ and any $l_1\ge 1$. This concludes the proof.

%

\end{proof}

%

%
%
%

In the case of weighted shifts on $\ell^p(\Z)$ or $\ell^p(\Z_+)$, we can generalize the equivalence obtained by Bayart and Ruzsa~\cite{Bayart2} by showing that every reiteratively hypercyclic weighted shift on $\ell^p(\Z)$ or $\ell^p(\Z_+)$ is frequently hypercyclic. The proof of the characterization obtained by Bayart and Ruzsa is based on the fact that if $A$ is a set with upper density $\overline{d}(A)=\delta>0$ and $\delta_k=\overline{d}(A\cap(A-k))$, then the set $F=\{k:\delta_k>(1-\varepsilon)\delta^2\}$ is syndetic~\cite[Theorem 8]{Bayart2}. We  remark that this result can be extended to sets with positive upper Banach density thanks to following two Theorems.
\begin{theorem}\emph{(Furstenberg Correspondence principle~\cite[Theorem 0.2 ]{PoSc})}
\label{FCP}
Given a subset $A\subset \Z$ of positive upper Banach density, there exists a measure-preserving system $(X, \mathfrak{B}, \mu, T)$ and a set $E\in \mathfrak{B}$ such that $\mu(E)=\overline{Bd}(A)$ and
\[
\overline{Bd}\big(A\cap (A-n_1)\cap \dots \cap (A-n_k) \big)\geq \mu \big(E\cap T^{-n_1}E\cap \dots \cap T^{-n_k}E\big)
\]
for any integer $n_1, \dots, n_k$.
\end{theorem}

\begin{theorem}\emph{(\cite[Theorem 3.1]{BD})}
\label{BDtheor}
For every measure-preserving system $(X, \mathfrak{B}, \mu, T)$, any $\varepsilon >0$ and $A\in \mathfrak{B}$ the set $\{n\in \Z: \mu(A\cap T^{-n}A)>\mu (A)^2-\varepsilon \}$ is syndetic.
\end{theorem}

We directly deduce from these theorems the following generalization of result of Bayart and Rusza.
\begin{theorem}
\label{variant.theor8.BaRu}
Let $A\subset \Z$ be a set with positive upper Banach density equal to $\delta$ and $\varepsilon\in ]0, 1[$. For any $k\in \Z$, let $B_k=A\cap (A-k)$ with upper Banach density equal to $\delta_k$. Then the set $\{k\in \Z: \delta_k>\delta^2-\varepsilon\}$ is syndetic.
\end{theorem}

In order to prove the equivalence between reiteratively hypercyclic weighted shifts and frequently hypercyclic weighted shifts on $\ell^p$, we need a little bit more precise result. This result and its proof are a direct adaptation of Theorem 8 in \cite{Bayart2} to sets with upper Banach density.  We include here the proof of this adaptation for the sake of completeness.

\begin{theorem}\label{lp1}
Let $A\subset \mathbb{Z_+}$ with $\overline{Bd}(A)=\delta>0$, $\varepsilon>0$ and $B_k=A\cap (A-k)$ where $k\ge 1$. Then there exist an increasing sequence $(m_i)_{i\ge 1}\subset \mathbb{N}$ and an increasing sequence $(k_i)_{i\ge 1}\subset \mathbb{N}$ such that
\begin{enumerate}
\item $\displaystyle{\frac{|A\cap[m_i,m_i+k_i[|}{k_i}\xrightarrow[i\to \infty]{} \delta}$;
\item $\displaystyle{\frac{|B_k\cap[m_i,m_i+k_i[|}{k_i}\xrightarrow[i\to \infty]{} \eta_k}$ for some $\eta_k$;
\item the set $F:=\{k\in \N:\eta_k>(1-\varepsilon)\delta^2\}$ is syndetic.
\end{enumerate}
\end{theorem}

\begin{proof}
Since $\overline{Bd}(A)=\delta$, there exists an increasing sequence $(n_i)_{i\ge 1}$ such that
\[\frac{|A\cap[n_i,n_i+i[|}{i}\xrightarrow[i\to \infty]{} \delta.\]
Since for any $k\ge 1$, the set $\big(\frac{|B_k\cap[n_i,n_i+i[|}{i}\big)_i$ belongs to the compact $[0,1]$, we can extract a subsequence $(n_{k_i})_{i\ge 1}$ such that for any $k\ge 1$, $\big(\frac{|B_k\cap[n_{k_i},n_{k_i}+k_i[}{k_i}\big)_i$ converges to some $\eta_k$. We denote by $(m_i)_{i\ge 1}$ the sequence $(n_{k_i})_{i\ge 1}$.

Let $R$ be a finite set such that for any $k,l\in R$, $k>l$, we have $\eta_{k-l}\le (1-\varepsilon)\delta^2.$
We consider $f(x):=|\{k\in R: x\in A-k\}|$. We deduce that for any $i\ge 1$
\begin{align*}
\sum_{x\in[m_i,m_i+k_i[}f(x)&=\sum_{k\in R}|(A-k)\cap [m_i,m_i+k_i[|\\
&=|R| |A\cap [m_i,m_i+k_i[|+ O(1).
\end{align*}
and thus
\[\frac{1}{k_i}\sum_{x\in[m_i,m_i+k_i[}f(x)\xrightarrow[i\to \infty]{}|R|\delta.\]
If we now consider the square of $f(x)$, we have
\[f(x)^2=|\{k,l\in R:x\in (A-k)\cap (A-l)\}|=|\{k,l\in R:x+k\in A\cap (A+k-l)\}|.\]
In the same way, we deduce that for any $i\ge 1$
\begin{align*}
&\sum_{x\in[m_i,m_i+k_i[}f(x)^2\\
\quad &=\sum_{k\in R}|A\cap [m_i,m_i+k_i[|+
2\sum_{k,l\in R,\ k>l}|B_{k-l}\cap [m_i,m_i+k_i[|+
O(1)\\
\quad &=|R||A\cap [m_i,m_i+k_i[|+
2\sum_{k,l\in R,\ k>l}|B_{k-l}\cap [m_i,m_i+k_i[|+ O(1)
\end{align*}
and thus
\[\frac{1}{k_i}\sum_{x\in[m_i,m_i+k_i[}f(x)^2\xrightarrow[i\to \infty]{}|R|\delta+2\sum_{k,l\in R,\ k>l}\eta_{k-l}\le |R|\delta+(1-\varepsilon)|R|(|R|-1)\delta^2.\]
We conclude that
\[(|R|\delta)^2\le |R|\delta+(1-\varepsilon)|R|(|R|-1)\delta^2.\]
This inequality gives us the following condition on $|R|$:
\[|R|\le \frac{1-\delta(1-\varepsilon)}{\delta\varepsilon}.\]
In other words, there exists a maximal finite set $R$ in the sense that for any $n\notin R$, there exists $k\in R$ such that $\eta_{n-k}>(1-\varepsilon)\delta^2$, i.e $n-k\in F$. We conclude that $F+R=\Z_+$ and in particular that $F$ is syndetic.
\end{proof}

Therefore, in the vein of \cite[Corollary 9]{Bayart2} and following essentially the same proof we obtain the following result:
\begin{corollary}\label{corbeta}
Let $A\subset \Z_+$ be a set with positive upper Banach density and $(\alpha_n)_{n\in \Z}$ a sequence of non-negative real numbers such that $\sum_n \alpha_n=+\infty$. Suppose that there exist some $C>0$ and $N\in \Z\cup\{+\infty\}$ such that $\alpha_{n}\geq C\alpha_{n-1}$ for every $n<N$ and $\alpha_{n}=0$ for every $n\ge N$. If for any $n\in A$, we let
\[
\beta_n=\sum_{m\in A}\alpha_{m-n},
\]
then the sequence $(\beta_n)_{n\in A}$ is not bounded.
\end{corollary}
\begin{proof}
Let $\delta:=\overline{Bd}(A)$ and $B_k=A\cap (A-k)$ for any $k\in \Z$. By Theorem~\ref{lp1}, we know that there exist an increasing sequence $(m_i)_{i\ge 1}\subset \mathbb{N}$ and an increasing sequence $(k_i)_{i\ge 1}\subset \mathbb{N}$ such that for any $k\ge 1$
\begin{enumerate}
\item $\displaystyle{\frac{|B_k\cap[m_i,m_i+k_i[|}{k_i}\xrightarrow[i\to \infty]{} \eta_k}$ for some $\eta_k$;
\item the set $F:=\{k\ge 1:\eta_k>\frac{1}{2}\delta^2\}$ is syndetic.
\end{enumerate}

We remark that for any $k\ge 1$, we have $B_{-k}=B_k+k$ and thus
\[ \frac{|B_{-k}\cap[m_i,m_i+k_i[|}{k_i}=\frac{|B_{k}\cap[m_i-k,m_i+k_i-k[|}{k_i}\xrightarrow[i\to \infty]{} \eta_k.\]

On the other hand, if we let $F_{\Z}=(-F)\cup F$, we deduce that the set $F_{\Z}$ is syndetic. Let $(f_j)_{j\in \Z}$ be an increasing enumeration of $F_{\Z}$ and $M\ge 1$ such that $f_j-f_{j-1}\le M$ for any $j\in \Z$. We deduce that for any $f_j<N$, we have
\[\alpha_{f_j}\ge \frac{\min(1,C^M)}{M}\sum_{f_{j-1}<i\le f_j}\alpha_i\]
and thus
\[\sum_{n\in F_{\Z}}\alpha_n=\sum_{j:f_j<N}\alpha_{f_j}\ge \frac{\min(1,C^M)}{M} \sum_{n<N-M}\alpha_n
=\frac{\min(1,C^M)}{M} \Big(\sum_{n\in \Z}\alpha_n -\sum_{n=N-M}^{N-1}\alpha_n\Big)=\infty,
\]
where we consider $\sum_{n=N-M}^{N-1}\alpha_n=0$ if $N=+\infty$.


We then consider the sequence $(s_i)_{i \ge 1}$ defined by
\[
s_i=\sum_{n\in A\cap[m_i,m_i+k_i[} \beta_n=\sum_{\substack{n\in A\cap[m_i,m_i+k_i[\\ m\in A}}\alpha_{m-n}.
\]
If we arrange this sum according to the value $k=m-n$ and if we keep only the terms where $k\in F_{\Z}$, then we get for any $l\ge 1$,
 \[
 s_i\geq \sum_{k\in F_{\Z}, \abs{k}< l}\alpha_k \abs{B_k \cap[m_i,m_i+k_i[}.
 \]
We deduce that for any $l\ge 1$
 \begin{align*}
 \limsup_{i\to \infty}\frac{1}{k_i} \sum_{n\in A\cap [m_i,m_i+k_i[} \beta_n&= \limsup_{i\to \infty}\frac{s_i}{k_i}\\
 &\ge \sum_{k\in F_{\Z}, \abs{k}< l}\alpha_k \lim_{i\to \infty}\frac{\abs{B_k\cap[m_i,m_i+k_i[}}{k_i}\\
 &\ge \sum_{k\in F_{\Z}, \abs{k}<l}\alpha_k \eta_{\abs{k}} \ge \frac{\delta^2}{2}\sum_{k\in F_{\Z}, \abs{k}<l}\alpha_k.
 \end{align*}
Since $\sum_{n\in F_{\Z}}\alpha_n=+\infty$, we conclude that the sequence $(\beta_n)_{n\in A}$ cannot be bounded.
\end{proof}

We can now prove the desired equivalence for weighted shifts on $\ell^p$.

\begin{theorem}\label{caraclp}
Let $B_w$ be a weighted shift on $\ell^p(\Z)$ or on $\ell^p(\Z_+)$ with $1\le p<\infty$.
Then $B_w$ is reiteratively hypercyclic if and only if $B_w$ is frequently hypercyclic.
\end{theorem}
\begin{proof}
We only prove this equivalence in the case of weighted shifts on $\ell^p(\Z)$ since the case of weighted shifts on  $\ell^p(\Z_+)$ is similar and easier.

Let $B_w$ be a reiteratively hypercyclic weighted shift on $\ell^p(\Z)$.
There exists a vector $x\in \ell^p(\Z_+)$ such that the set
\[A:=\left\{n\in \Z_+: \|B_w^nx-e_0\|\le \frac{1}{2}\right\}\]
has a positive upper Banach density. For any $n\in A$, we remark that we have
$|w_1\cdots w_n x_n-1|\le \frac{1}{2}$ and
\begin{align*}
\frac{1}{2^p}&\ge \sum_{m<n}|w_{m-n+1}\cdots w_0w_1\cdots w_{m}|^p|x_m|^p+ \sum_{m>n}|w_{m-n+1}\cdots w_{m}|^p|x_m|^p\\
&= \sum_{m<n}|w_{m-n+1}\cdots w_0|^p|w_1\cdots w_{m}x_m|^p+ \sum_{m>n}\frac{|w_1\cdots w_mx_m|^p}{|w_1\cdots w_{m-n}|^p}\\
&\ge \sum_{m<n,m\in A}|w_{m-n+1}\cdots w_0|^p|w_1\cdots w_{m}x_m|^p+ \sum_{m>n, m\in A}\frac{|w_1\cdots w_mx_m|^p}{|w_1\cdots w_{m-n}|^p}\\
&\ge \frac{1}{2^p}\left(\sum_{m<n,m\in A}|w_{m-n+1}\cdots w_0|^p+ \sum_{m>n, m\in A}\frac{1}{|w_1\cdots w_{m-n}|^p}\right).
\end{align*}
We get for any $n\in A$
\begin{equation}
\sum_{m<n,m\in A}|w_{m-n+1}\cdots w_0|^p\le 1 \quad\text{and}\quad \sum_{m>n, m\in A}\frac{1}{|w_1\cdots w_{m-n}|^p}\le 1.
\label{eqbeta}
\end{equation}
Thanks to Corollary~\ref{corbeta}, we can deduce from \eqref{eqbeta} the convergence of series $\sum_{n\ge 1}\frac{1}{|w_{0}\cdots w_{n}|^p}$ and $\sum_{n<0}|w_{n}\cdots w_{0}|^p$. Indeed, if we let $\alpha_n=0$ for any $n\le 0$ and $\alpha_n=\frac{1}{|w_1\cdots w_n|^p}$ for any $n\ge 1$, we have $\alpha_n\ge C\alpha_{n-1}$ where $C=\inf\{|w_n|^{-p}:n\ge 1\}$ is strictly positive since $w$ is bounded. Therefore, if $\sum_{n\ge 1}\frac{1}{|w_{1}\cdots w_{n}|^p}=\infty$, we deduce from Corollary~\ref{corbeta} with $N=+\infty$ that the sequence $(\beta_{n})_{n\in A}$ is unbounded where
\[\beta_n:=\sum_{m\in A}\alpha_{m-n}=\sum_{m>n,m\in A}\alpha_{m-n}=\sum_{m>n, m\in A}\frac{1}{|w_1\cdots w_{m-n}|^p}.\]
This is a contradiction with~$\eqref{eqbeta}$.

On the other hand, if we let $\alpha_n=0$ for any $n\ge 0$ and $\alpha_{n}=|w_{n+1}\cdots w_0|^p$ for any $n<0$, we have $\alpha_{n}\ge C\alpha_{n-1}$ for any $n<0$ where $C=\inf\{|w_n|^{-p}:n\le -1\}$. As previously, we deduce from Corollary~\ref{corbeta} with $N=0$ that if $\sum_{n<0}|w_{n}\cdots w_{0}|^p=\infty$ then the sequence $(\sum_{m<n,m\in A}|w_{m-n+1}\cdots w_0|^p)_{n\in A}$ is unbounded which is a contradiction with~$\eqref{eqbeta}$.

Finally, we get the desired result since the convergence of series $\sum_{n\ge 1}\frac{1}{|w_{0}\cdots w_{n}|^p}$ and $\sum_{n<0}|w_{n}\cdots w_{0}|^p$ implies that $B_w$ is frequently hypercyclic (\cite[Theorem~3]{Bayart2})
\end{proof}

This equivalence for weighted shifts on $\ell^p(\Z_+)$ implies that there exists some mixing operator which is not reiteratively hypercyclic. Indeed, we know that a weighted shift $B_w$ on $\ell^p(\Z_+)$ is mixing if and only if $\prod_{k=1}^{\infty}|w_k|$ tends to infinity~\cite{Costakis}. Therefore, the weighted backward shift $B_w$ with $w_n=\big((n+1)/n)^{\frac{1}{p}}$ is a mixing operator on $\ell^p(\Z_+)$ which is not frequently hypercyclic and thus not reiteratively hypercyclic (Theorem~\ref{caraclp}).

\begin{theorem}
There exists some mixing operator which is not reiteratively hypercyclic.
\end{theorem}

Thanks to Grosse-Erdmann and Peris~\cite{Grosse0}, we know that every frequently hypercyclic operator is weakly mixing. We observe that every reiteratively hypercyclic operator is  topologically ergodic (thus, weakly mixing \cite{GEP10}). See also \cite{BG07}.

\begin{proposition}\label{wm}
Let $X$ be a separable $F$-space and $T\in \mathcal{L}(X)$.
If $T$ is reiteratively hypercyclic, then $T$ is topologically ergodic.
\end{proposition}

\begin{proof}
Let $U, V$ non-empty open sets in $X$ and $n\in N(U, V)$. We consider the non-empty open set $U_n:=U\cap T^{-n}(V)$. Let $x\in X$ such that $\overline{Bd}\left(N(x, U_n)\right)>0$.
We remark that
 \begin{equation*}
 N(x, U_n)-N(x, U_n)+n\subseteq N(U, V).
 \end{equation*}
Indeed, if $s_1, s_2\in N(x, U_n)$, then
\[
 T^{s_2}x\in U\quad\text{and}\quad T^{s_1-s_2+n}(T^{s_2}x)=T^{n}(T^{s_1}x)\in V
\]
On the other hand, we know that if $A$ is a set with positive upper Banach density, then $A-A$ is syndetic~\cite[Proposition 3.19]{Furstenberg}. We conclude that $N(x, U_n)-N(x, U_n)$ is syndetic and thus $N(U, V)$ is also syndetic.
\end{proof}


\begin{thebibliography}{99}

 \bibitem{BG07} C. Badea and S. Grivaux, \textit{Unimodular eigenvalues, uniformly distributed sequences and linear dynamics}, Advances in Math. 211 (2007) 766-793.
  	
  	\bibitem{Bayart1} F. Bayart and S. Grivaux, \textit{Frequently hypercyclic operators}, Trans. Amer. Math. Soc. 358 (2006), 5083-5117.
  	
  	\bibitem{2Bayart1} F. Bayart and S. Grivaux, \emph{Invariant Gaussian measures for operators on Banach spaces and linear dynamics}, Proc. Lond. Math. Soc. (3) 94 (2007), 181--210.
  	
  	\bibitem{BM09} F. Bayart and \'{E}. Matheron, \textit{Dynamics of linear operators}, Cambridge Tracts in Mathematics, No. 179, 2009.
  	
  	\bibitem{Bayart3} F. Bayart and \'{E}. Matheron, \emph{(Non-)weakly mixing operators and hypercyclicity sets}, Ann. Inst. Fourier (Grenoble) 59 (2009), 1--35.
  	
  	\bibitem{Bayart2} F. Bayart and I. Ruzsa, \textit{Difference sets and frequently hypercyclic weighted shifts}, Ergod. Th. \& Dynam. Sys. (2013), available on CJO2013. doi:10.1017/etds.2013.77.
  	
  	
	\bibitem{BD} V. Bergelson, T. Downarowicz, \textit{Large sets of integers and hierarchy of mixing properties of measure-preserving systems}, Colloq. Math. 110 (2008), no. 1, 117-150.
  	
  	\bibitem{Bernal} L. Bernal-Gonz\'{a}lez and K.-G. Grosse-Erdmann, \emph{The Hypercyclicity Criterion for sequences of operators}, Studia Math. 157 (2003), 17--32.
  	
  	\bibitem{Bes} J. B\`{e}s and A. Peris, \emph{Hereditarily hypercyclic operators}, J. Funct. Anal. 167 (1999), 94--112.
  	
  	\bibitem{4Bonilla} A. Bonilla and K.-G. Grosse-Erdmann, \emph{Frequently hypercyclic operators and vectors}, Ergodic Theory Dynam. Systems 27 (2007), 383--404. Erratum: Ergodic Theory Dynam. Systems 29 (2009), 1993--1994.  	
  	
  	\bibitem{Costakis} G. Costakis and M. Sambarino, \emph{Topologically mixing hypercyclic operators}, Proc. Amer. Math. Soc. 132 (2004), 385--389.
  	
  	\bibitem{Furstenberg} H. Furstenberg, \textit{Recurrence in ergodic Theory and combinatorial number Theory}, Princeton university press, Princeton, N.J., 1981.
	
	\bibitem{2Grosse3} K.-G. Grosse-Erdmann \emph{Hypercyclic and chaotic weighted shifts}, Studia Math. 139 (2000), 47--68.
	
 	\bibitem{Grosse0} K.-G. Grosse-Erdmann and A. Peris, \emph{Frequently dense orbits}, C. R. Math. Acad. Sci. Paris 341 (2005), 123--128.
 	
	
	\bibitem{GEP10} K.-G.~Grosse-Erdmann and A. Peris, \textit{Weakly mixing operators on topological vector spaces}, RACSAM 104 (2), 2010, 413-426.

	\bibitem{GEP11} K.-G.~Grosse-Erdmann and A. Peris Manguillot, \textit{Linear chaos}, Universitext, Springer, London, 2011.	
	
	\bibitem{PoSc} M. Pollicott, K. Schmidt, \textit{Ergodic theory of $\Z^d$-actions}, London Math. Soc. Lecture Note Series 228 .

	\bibitem{2Salas} H. N. Salas, \emph{Hypercyclic weighted shifts}, Trans. Amer. Math. Soc. 347 (1995), 993--1004.

	\bibitem{Salat} T. Salat and V. Toma, \textit{A classical Olivier's theorem and statistical convergence}, Ann. Math. Blaise Pascal 10 (2003), 305--313.
	
	\bibitem{Shkarin} S. Shkarin, \textit{On the spectrum of frequently hypercyclic operators},  Proc. Amer. Math. Soc.  137  (2009), 123-134.

\end{thebibliography}
\end{document}